\documentclass[a4paper]{amsart}

\usepackage{amsmath,amssymb,amsthm}
\usepackage[all]{xy}
\usepackage[colorlinks=true,backref=page]{hyperref}

\newtheorem{theorem}{Theorem}[section]
\newtheorem{proposition}[theorem]{Proposition}
\newtheorem{lemma}[theorem]{Lemma}

\theoremstyle{definition}

\theoremstyle{remark}
\newtheorem{remark}[theorem]{Remark}

\numberwithin{equation}{section}

\newcommand{\Z}{\mathbb{Z}}
\newcommand{\Q}{\mathbb{Q}}
\newcommand{\R}{\mathbb{R}}
\newcommand{\C}{\mathbb{C}}
\renewcommand{\H}{\mathbb{H}}
\renewcommand{\O}{\mathbb{O}}
\newcommand{\E}{\mathrm{E}}
\newcommand{\F}{\mathrm{F}}
\newcommand{\G}{\mathrm{G}}
\newcommand{\U}{\mathrm{U}}
\newcommand{\SU}{\mathrm{SU}}
\newcommand{\SO}{\mathrm{SO}}
\newcommand{\Sp}{\mathrm{Sp}}
\newcommand{\PSp}{\mathrm{PSp}}
\newcommand{\Spin}{\mathrm{Spin}}

\newcommand{\Sq}{\operatorname{Sq}}

\SelectTips{cm}{}

\title[Homotopy commutativity in symmetric spaces]{Homotopy commutativity in symmetric spaces}

\author[Daisuke Kishimoto]{Daisuke Kishimoto}
\address{Faculty of Mathematics, Kyushu University, Fukuoka 819-0395, Japan}
\email{kishimoto@math.kyushu-u.ac.jp}

\author[Yuki Minowa]{Yuki Minowa}
\address{Department of Mathematics, Kyoto University, Kyoto, 606-8502, Japan}
\email{minowa.yuki.48z@st.kyoto-u.ac.jp}

\author[Toshiyuki Miyauchi]{Toshiyuki Miyauchi}
\address{Department of Applied Mathematics, Faculty of Science, Fukuoka University, Fukuoka 814-0180, Japan}
\email{miyauchi@math.sci.fukuoka-u.ac.jp}

\author[Yichen Tong]{Yichen Tong}
\address{Department of Mathematics, Kyoto University, Kyoto 606-8502, Japan}
\email{tong.yichen.25m@st.kyoto-u.ac.jp}

\date{\today}



\begin{document}
	
	\maketitle
	
	\begin{abstract}
		We extend the former results of Ganea and the two of the authors with Takeda on the homotopy commutativity of the loop spaces of Hermitian symmetric spaces such that the loop spaces of all irreducible symmetric spaces but $\C P^3$ are not homotopy commutative.
	\end{abstract}

	%%%%% Section 1 %%%%%
	
	\section{Introduction}
	
	It is an interesting problem to determine whether or not a given H-space is homotopy commutative. The most celebrated result on this problem is Hubbuck's torus theorem \cite{H} which states that a connected finite H-space is homotopy commutative if and only if it is homotopy equivalent to a torus. For studying this problem for infinite H-spaces, we need to fix a particular class, as there are vast classes of infinite H-spaces, each of which has its special feature.

	We consider the loop space of a simply-connected finite complex, which is an infinite H-space whenever the underlying finite complex is non-contractible. Ganea \cite{Ga} studied the homotopy nilpotency, including the homotopy commutativity, of the loop space of a complex projective space, and showed that the loop space of $\C P^n$ is homotopy commutative if and only if $n=3$. Recently, Golasi\'{n}ski \cite{Go} studied the loop space of a homogeneous space, and in particular, he proved that complex Grassmannians are homotopy nilpotent. However, from this result, we cannot deduce the homotopy commutativity in most cases.

	Recall that every symmetric space decomposes into a product of irreducible symmetric spaces, and as in \cite{Ca}, irreducible symmetric spaces are classified such that they are the homogeneous spaces $G/H$ in Table \ref{table} below. Then complex projective spaces and complex Grassmanians are the irreducible symmetric spaces $\mathrm{AIII}$. The symmetric spaces $\mathrm{AIII}$, $\mathrm{BDI}$, $\mathrm{CI}$, $\mathrm{DIII}$, $\mathrm{EIII}$, $\mathrm{EVII}$ are called irreducible Hermitian symmetric spaces, which are exactly symmetric spaces having almost complex structures. Continuing the works of Ganea \cite{Ga} and Golasi\'{n}ski \cite{Go}, two of the authors and Takeda \cite{KTT} studied the homotopy commutativity of the loop spaces of Hermitian symmetric spaces, and obtained that the loop spaces of all irreducible Hermitian symmetric spaces but $\C P^3$ are not homotopy commutative.

	In this paper, we study the homotopy commutativity of all irreducible symmetric spaces, and extend the above result on Hermitian symmetric spaces as:

	\begin{theorem}
		\label{main}
		The loop spaces of all irreducible symmetric spaces but $\C P^3$ are not homotopy commutative.
	\end{theorem}

	As in Lemma \ref{non-commutative} below, if we can find a non-trivial Whitehead product in a path-connected space $X$, then we can deduce that the loop space of $X$ is not homotopy commutative. In \cite{KTT}, the existence of a non-trivial Whitehead product is proved by applying criteria in terms of rational homotopy theory and Steenrod operations. In this paper, these criteria will be employed too, where the rational homotopy criterion will be elaborated. However, these criteria are not applicable to the symmetric spaces $\mathrm{AII}$ and $\mathrm{EIV}$, where they are H-spaces when localized at any odd prime. Then we will develop a new cohomological technique for showing the existence of a non-trivial Whitehead product, which applies to $\mathrm{AII}$ and $\mathrm{EIV}$.

	\subsection*{Acknowledgement}
	
	The authors were partially supported by JSPS KAKENHI Grant Number JP17K05248 and JP1903473 (Kishimoto), JST, the establishment of university fellowships towards the creation of science technology innovation, Grant Number JPMJFS2123 (Minowa), and JST SPRING Grant Number JPMJSP2110 (Tong).

	\renewcommand{\arraystretch}{1.3}
	
	\begin{table}[htbp]
		\label{table}
		\centering
		\caption{Irreducible symmetric spaces}
		\begin{tabular}{|l|lll|}
			\hline
			&$G$&$H$&\\\hline
			$\mathrm{AI}$&$\SU(n)$&$\SO(n)$&($n\ge 2$)\\
			$\mathrm{AII}$&$\SU(2n)$&$\Sp(n)$&($n\ge 2$)\\
			$\mathrm{AIII}$&$\U(m+n)$&$\U(m)\times\U(n)$&($m,n\ge 1$)\\
			$\mathrm{BDI}$&$\SO(m+n)$&$\SO(m)\times\SO(n)$&($m,n\ge 2$)\\
			$\mathrm{DIII}$&$\SO(2n)$&$\U(n)$&($n\ge 2$)\\
			$\mathrm{CI}$&$\Sp(n)$&$\U(n)$&($n\ge 2$)\\
			$\mathrm{CII}$&$\Sp(m+n)$&$\Sp(m)\times\Sp(n)$&($m,n\ge 1$)\\
			$\mathrm{EI}$&$\E_6$&$\PSp(4)$&\\
			$\mathrm{EII}$&$\E_6$&$\SU(6)\times \SU(2)$&\\
			$\mathrm{EIII}$&$\E_6$&$\Spin(10)\cdot S^1$&($\Spin(10)\cap S^1\cong\Z_4$)\\
			$\mathrm{EIV}$&$\E_6$&$\F_4$&\\
			$\mathrm{EV}$&$\E_7$&$\SU(8)/\{\pm I\}$&\\
			$\mathrm{EVI}$&$\E_7$&$\Spin(12)\cdot\SU(2)$&($\Spin(12)\cap\SU(2)\cong\Z_2$)\\
			$\mathrm{EVII}$&$\E_7$&$\E_6\cdot S^1$&($E_6\cap S^1\cong\Z_3$)\\
			$\mathrm{EVIII}$&$\E_8$&$\mathrm{Ss}(16)$&\\
			$\mathrm{EIX}$&$\E_8$&$\E_7\cdot\SU(2)$&($E_7\cdot\SU(2)\cong\Z_2$)\\
			$\mathrm{FI}$&$\F_4$&$\Sp(3)\cdot\Sp(1)$&($\Sp(3)\cap\Sp(1)\cong\Z_2$)\\
			$\mathrm{FII}$&$\F_4$&$\Spin(9)$&\\
			$\mathrm{G}$&$\G_2$&$\SO(4)$&\\\hline
		\end{tabular}
	\end{table}

	%%%%% Section 2 %%%%%
	
	\section{Rational homotopy}
	
	In this section, we consider the existence of a non-trivial Whitehead product by using rational homotopy theory, and apply it to the symmetric spaces $\mathrm{EII}$, $\mathrm{EV}$, $\mathrm{EVI}$, $\mathrm{EVIII}$, $\mathrm{EIX}$, $\mathrm{FI}$. We start by showing a sufficient condition for a loop space not being homotopy commutative.

	\begin{lemma}
		\label{non-commutative}
		Let $X$ be a path-connected space. If there is a non-trivial Whitehead product in $X$, then the loop space of $X$ is not homotopy commutative.
	\end{lemma}
	
	\begin{proof}
		By the adjointness of Whitehead products and Samelson products, if there is a non-trivial Whitehead product in $X$, then there is a non-trivial Samelson product in $\Omega X$, implying that the loop space of $X$ is not homotopy commutative.
	\end{proof}

	We consider rational homotopy theory. For a positively graded vector space $V$ over $\Q$, let $\Lambda V$ denote the free commutative graded algebra generated by $V$. The following lemma follows from \cite[Proposition 13.16]{FHT}.

	\begin{lemma}
		\label{rational Whitehead product}
		Let $(\Lambda V,d)$ be the minimal Sullivan model for a simply-connected space $X$ of finite rational type. If there is $x\in V$ such that $dx$ is decomposable and includes the term $yz$ for $0\ne y,z\in V$, then there are $\alpha\in\pi_m(X)\otimes\Q$ and $\beta\in\pi_n(X)\otimes\Q$ such that the Whitehead product $[\alpha,\beta]$ is non-trivial in $\pi_{m+n-1}(X)\otimes\Q$, where $|y|=m$ and $|z|=n$.
	\end{lemma}

	%Recall that a simply-connected space $X$ is formal if its rational homotopy type is a formal consequence of the rational cohomology. More precisely, there is a zig-zag of quasi-isomorphisms of differential graded algebras between minimal Sullivan model for $X$ and the rational cohomology of $X$. The following lemma is proved in \cite{Ca}.

	%\begin{proposition}
	%  \label{symmetric space formal}
	%  Symmetric spaces are formal.
	%\end{proposition}

	%We consider the minimal Sullivan model for a special formal space, which together with Lemma \ref{rational Whitehead product} shows the existence of a non-trivial Whitehead product.

	The following lemma enables us to apply Lemma \ref{rational Whitehead product}, in the special case, only by looking at rational cohomology.

	\begin{lemma}
		\label{formal minimal model}
		Let $X$ be a simply-connected finite complex such that
		\[
		H^*(X;\Q)=\Q[x_1,\ldots,x_n]/(\rho_1,\ldots,\rho_n)
		\]
		where $|x_1|,\ldots,|x_n|$ are even and all $\rho_1,\ldots,\rho_n$ are decomposable. Then the minimal Sullivan model for $X$ is given by
		\[
		\Lambda(x_1,\ldots,x_n,y_1,\ldots,y_n),\quad dx_i=0,\quad dy_i=\rho_i.
		\]
	\end{lemma}
	
	\begin{proof}
		Since $X$ is a finite complex, $H^*(X;\Q)$ is a finite dimensional vector space. Then $\rho_1,\ldots,\rho_n$ is a regular sequence. Let $A$ be the differential graded algebra in the statement. Since $\rho_1,\ldots,\rho_n$ are decomposable, $A$ is a minimal Sullivan algebra. Let $I_k$ denote the degree $>k$ part of $\Lambda(x_1,\ldots,x_n)$, and let
		\[
		A_k=(\Lambda(x_1,\ldots,x_n)/I_k)\otimes\Lambda(y_1,\ldots,y_n).
		\]
		Then we get a sequence $0=A_{-1}\leftarrow A_0\leftarrow A_1\leftarrow\cdots$ of surjections which yields a spectral sequence $E_r$ converging to $H^*(A)$. It is easy to see that each $y_i$ is transgressive and the transgression image of $y_i$ is $\rho_i$. Then since $\rho_1,\ldots,\rho_n$ is a regular sequence, we get an isomorphism
		\[
		E_\infty=E_\infty^{*,0}\cong H^*(X;\Q).
		\]
		Since $E_\infty=E_\infty^{*,0}$, the extension problem is trivial, so $H^*(A)\cong H^*(X;\Q)$. Thus by \cite[Lemma 4.2]{MN}, $X$ is formal.

		There is a map $f\colon A\to H^*(A)$ of differential graded algebras, where $f(x_i)=x_i$ and $f(y_i)=0$. Clearly, the induced map $f^*\colon H^*(A)\to H^*(A)$ is surjective, hence an isomorphism because $H^*(A)\cong H^*(X;\Q)$ is a finite dimensional vector space. Thus since $X$ is formal, the differential graded algebra $A$ is the minimal Sullivan model for $X$, completing the proof.
	\end{proof}

	Now we consider the homotopy commutativity of symmetric spaces $\mathrm{EII}$, $\mathrm{EV}$, $\mathrm{EVI}$, $\mathrm{EVIII}$, $\mathrm{EIX}$, $\mathrm{FI}$ by applying Lemmas \ref{rational Whitehead product} and \ref{formal minimal model}.

	\begin{proposition}
		\label{EII}
		The loop space of $\mathrm{EII}$ is not homotopy commutative.
	\end{proposition}
	
	\begin{proof}
		In \cite{I2}, Ishitoya determined the integral cohomology of $\mathrm{EII}$. In particular, we have
		\[
		H^*(\mathrm{EII};\Q)=\Q[x_4,x_6,x_8]/(\rho_{16},\rho_{18},\rho_{24}),\quad|x_i|=|\rho_i|=i
		\]
		where $\rho_{16},\rho_{18},\rho_{24}$ are decomposable and $\rho_{16}$ includes the term $x_8^2$. Then by Lemmas \ref{rational Whitehead product} and \ref{formal minimal model}, there is a non-trivial Whitehead product in $\mathrm{EII}$, so by Lemma \ref{non-commutative}, the proof is finished.
	\end{proof}

	\begin{proposition}
		\label{EV-EVIII}
		The loop spaces of $\mathrm{EV}$ and $\mathrm{EVIII}$ are not homotopy commutative.
	\end{proposition}
	
	\begin{proof}
		Let $G$ and $H$ be as in Table \ref{table} for $\mathrm{EV}$ and $\mathrm{EVIII}$. Then there is an isomorphism
		\[
		H^*(G/H;\Q)\cong H^*(BT;\Q)^{W(H)}/(\widetilde{H}^*(BT;\Q)^{W(G)})
		\]
		where $T$ is a common maximal torus of $G,H$ and $W(G),W(H)$ denote the Weyl groups of $G,H$. So by \cite{HKMO,W}, we get
		\begin{align*}
			H^*(\mathrm{EV};\Q)&=\Q[x_6,x_8,x_{10},x_{14}]/(\rho_{20},\rho_{24},\rho_{28},\rho_{36})\\
			H^*(\mathrm{EVIII};\Q)&=\Q[y_8,y_{12},y_{16},y_{20}]/(\mu_{36},\mu_{40},\mu_{48},\mu_{60})
		\end{align*}
		where $|x_i|=|y_i|=|\rho_i|=|\mu_i|=i$, all $\rho_i,\mu_i$ are decomposable, and $\rho_{20},\mu_{40}$ include the terms $x_{10}^2,y_{20}^2$, respectively. Hence by Lemmas \ref{rational Whitehead product} and \ref{formal minimal model}, there are non-trivial Whitehead products in $\mathrm{EV}$ and $\mathrm{EVIII}$. Thus by Lemma \ref{non-commutative}, the proof is finished.
	\end{proof}

	\begin{proposition}
		\label{EVI-EIX-FI}
		The loop spaces of $\mathrm{EVI},\,\mathrm{EIX},\,\mathrm{FI}$ are not homotopy commutative.
	\end{proposition}
	
	\begin{proof}
		We consider $\mathrm{FI}=\F_4/\Sp(3)\cdot\Sp(1)$. In \cite{IT}, the integral cohomology of $\F_4/\Sp(3)\cdot S^1$ is determined, and in particular, we have
		\[
		H^*(\F_4/\Sp(3)\cdot S^1;\Q)=\Q[x_2,x_8]/(\rho_{16},\rho_{24}),\quad|x_i|=|\rho_i|=i
		\]
		where $\rho_{16}$ includes the term $x_8^2$. Then by Lemmas \ref{rational Whitehead product} and \ref{formal minimal model}, there are $\alpha,\beta\in\pi_8(\F_4/\Sp(3)\cdot S^1)\otimes\Q$ such that $[\alpha,\beta]\ne 0$ in $\pi_{15}(\F_4/\Sp(3)\cdot S^1)\otimes\Q$.

		Consider the homotopy exact sequence for a fibration
		\[
		\Sp(1)/S^1=S^2\to\F_4/\Sp(3)\cdot S^1\xrightarrow{q}\mathrm{FI}.
		\]
		Then since $\pi_*(S^2)\otimes\Q=0$ for $*\ge 4$, the induced map $q_*\colon\pi_*(\mathrm{F}_4/\Sp(3)\cdot S^1)\otimes\Q\to\pi_*(\mathrm{FI})\otimes\Q$ is an isomorphism for $*\ge 5$. So there are $\tilde{\alpha},\tilde{\beta}\in\pi_8(\mathrm{FI})\otimes\Q$ such that $q_*(\tilde{\alpha})=\alpha$ and $q_*(\tilde{\beta})=\beta$, hence we get
		\[
		q_*([\tilde{\alpha},\tilde{\beta}])=[q_*(\tilde{\alpha}),q_*(\tilde{\beta})]=[\alpha,\beta].
		\]
		Thus we obtain $[\tilde{\alpha},\tilde{\beta}]\ne 0$, and by Lemma \ref{non-commutative}, the loop space of $\mathrm{FI}$ is not homotopy commutative.

		By \cite{N1,N2}, we have
		\begin{align*}
			H^*(\E_7/\Spin(12)\cdot S^1;\Q)&=\Q[x_2,x_8,x_{12}]/(\rho_{24},\rho_{28},\rho_{36})\\
			H^*(\E_8/\E_7\cdot S^1;\Q)&=\Q[y_2,y_{12},y_{20}]/(\mu_{40},\mu_{48},\mu_{60})
		\end{align*}
		where $|x_i|=|y_i|=|\rho_i|=|\mu_i|=i$, all $\rho_i,\mu_i$ are decomposable, and $\rho_{24},\mu_{40}$ include the terms $x_{12}^2,y_{20}^2$, respectively. Then quite similarly to the $\mathrm{FI}$ case, we can see that the loop spaces of $\mathrm{EVI}$ and $\mathrm{EIX}$ are not homotopy commutative, completing the proof.
	\end{proof}

	%%%%% Section 3 %%%%%
	
	\section{Steenrod operations}
	
	In this section, we use the Steenrod operations to show the existence of a non-trivial Whitehead product, and apply this criterion to $\mathrm{AI}$, $\mathrm{BDI}$, $\mathrm{CII}$, $\mathrm{EI}$, $\mathrm{FII}$, $\G$. The following lemma is proved in \cite{KTT}. Let $QH^*(X)$ denote the module of indecomposables of $H^*(X)$.

	\begin{lemma}
		\label{Steenrod operation}
		Suppose that maps $\alpha\colon\Sigma A\to X$ and $\beta\colon\Sigma B\to X$ satisfy the following conditions:
		\begin{enumerate}
			\item there are $a,b\in H^*(X;\Z_p)$ such that $\alpha^*(a)\ne 0$ and $\beta^*(b)\ne 0$;
			
			\item $\beta^*(a)=0$ for $p=2$;
			
			\item $A=B$, $\alpha=\beta$, $a=b$ for $|a|=|b|$ and $p$ odd;
			
			\item $\dim QH^k(X;\Z_p)=1$ for $|a|=k$;
			
			\item there is $x\in H^*(X;\Z_p)$ and a cohomology operation $\theta$ such that $\theta(x)$ is decomposable and includes the term $ab$;
			
			\item $\theta(H^n(\Sigma A\times\Sigma B;\Z_p))=0$ for $|x|=n$.
		\end{enumerate}
		Then the Whitehead product $[\alpha,\beta]$ is non-trivial.
	\end{lemma}

	\subsection{$\mathrm{AI}$, $\mathrm{BDI}$, and $\mathrm{CII}$}
	
	To apply Lemma \ref{Steenrod operation} to $\mathrm{AI}$ and $\mathrm{BDI}$, we consider the map
	\[
	g\colon\R P^{n-1}\to\SO(n),\quad[x_1:\cdots:x_n]\mapsto (I-2X)\mathrm{diag}(-1,1,\ldots,1)
	\]
	defined in \cite{Wh}, where
	\[
	X=\frac{1}{x_1^2+\cdots+x_n^2}
	\begin{pmatrix}
		x_1x_1&x_1x_2&\cdots&x_1x_n\\
		x_2x_1&x_2x_2&\cdots&x_2x_n\\
		\vdots&\vdots&&\vdots\\
		x_nx_1&x_nx_2&\cdots&x_nx_n\\
	\end{pmatrix}
	\]
	and $\mathrm{diag}(a_1,\ldots,a_n)$ stands for the diagonal matrix with diagonal entries $a_1,\ldots,a_n$.
	Note that $I-2X$ is the reflection in the hyperplane through the origin associated to $(x_1,\ldots,x_n)\in\R^n$.

	\begin{lemma}
		\label{reflection}
		Let $c\colon\SO(n)\to\SU(n)$ denote the inclusion. Then the composite
		\[
		\R P^{n-1}\xrightarrow{g}\SO(n)\xrightarrow{c}\SU(n)
		\]
		is null-homotopic.
	\end{lemma}
	
	\begin{proof}
		As in \cite{Y}, we consider the map $h\colon\Sigma\C P^{n-1}\to\SU(n)$ defined by
		\begin{align*}
			&h([z_1:\cdots:z_n],t)=\\
			&\left(I-2e^{-\pi\sqrt{-1}\left(t-\frac{1}{2}\right)}\cos\pi\left(t-\frac{1}{2}\right)Z\right)\mathrm{diag}\left(-e^{2\pi\sqrt{-1}\left(t-\frac{1}{2}\right)},1,\ldots,1\right)
		\end{align*}
		where
		\[
		Z=\frac{1}{|z_1|^2+\cdots+|z_n|^2}
		\begin{pmatrix}
			z_1\overline{z_1}&z_1\overline{z_2}&\cdots&z_1\overline{z_n}\\
			z_2\overline{z_1}&z_2\overline{z_2}&\cdots&z_2\overline{z_n}\\
			\vdots&\vdots&&\vdots\\
			z_n\overline{z_1}&z_n\overline{z_2}&\cdots&z_n\overline{z_n}\\
		\end{pmatrix}.
		\]
		Then for the inclusion
		\[
		i\colon\R P^{n-1}\to\Sigma\C P^{n-1},\quad[x_1,\ldots,x_n]\mapsto\left([x_1,\ldots,x_n],\frac{1}{2}\right)
		\]
		there is a commutative diagram
		\[
		\xymatrix{
			\R P^{n-1}\ar[r]^g\ar[d]&\SO(n)\ar[d]^c\\
			\Sigma\C P^{n-1}\ar[r]^h&\SU(n).
		}
		\]
		Thus since the inclusion $i\colon\R P^{n-1}\to\Sigma\C P^{n-1}$ is null-homotopic, the composite $c\circ g$ is null-homotopic too.
	\end{proof}

	Now we are ready to apply Lemma \ref{Steenrod operation} to $\mathrm{AI}$ and $\mathrm{BDI}$.
	\begin{proposition}
		\label{AI}
		The loop space of $\mathrm{AI}=\SU(n)/\SO(n)$ for $n\ge 2$ is not homotopy commutative.
	\end{proposition}
	
	\begin{proof}
		Since $\SU(2)/\SO(2)=S^2$ and $[1_{S^2},1_{S^2}]\ne 0$, it follows from Lemma \ref{non-commutative} that the loop space of $\SU(2)/\SO(2)$ is not homotopy commutative. Then we assume $n\ge 3$. Let $\iota\colon\SU(n)/\SO(n)\to B\SO(n)$ denote the natural map. By \cite[Theorem 6.7]{MT}, the mod $2$ cohomology of $\SU(n)/\SO(n)$ is given by
		\[
		H^*(SU(n)/SO(n);\Z_2)=\Lambda(v_2,v_3,\ldots,v_n),\quad v_i=\iota^*(w_i)
		\]
		where $w_i\in H^i(B\SO(n);\Z_2)$ denotes the $i$-th Stiefel-Whitney class. Let $\bar{g}\colon\Sigma\R P^{n-1}\to BSO(n)$ denote the adjoint of the map $g\colon\R P^{n-1}\to\SO(n)$. Then by Lemma \ref{reflection}, the composite
		\[
		\Sigma\R P^{n-1}\xrightarrow{\bar{g}}B\SO(n)\xrightarrow{c}B\SU(n)
		\]
		is null-homotopic, so the map $\bar{g}$ lifts to a map $\tilde{g}\colon\Sigma\R P^{n-1}\to\SU(n)/\SO(n)$ through $\iota$, up to homotopy, because there is a homotopy fibration
		\[
		\SU(n)/\SO(n)\xrightarrow{\iota}B\SO(n)\xrightarrow{c}B\SU(n).
		\]
		By \cite{Wh}, we have $\bar{g}^*(w_i)=\Sigma u^{i-1}$ for $i=2,\ldots,n$, where $u$ is a generator of $H^1(\R P^{n-1};\Z_2)\cong\Z_2$. Then we get
		\[
		\tilde{g}^*(v_i)=\Sigma u^{i-1}.
		\]

		\noindent(1) Suppose $n\equiv 0,3\mod 4$. Let $h=\tilde{g}\vert_{\Sigma\R P^1}$. Clearly, we have $h^*(v_n)=0$ and $h^*(v_2)=\Sigma u$. By the Wu formula, we have $\Sq^2w_n=w_2w_n$, implying
		\[
		\Sq^2v_n=v_2v_n.
		\]
		On the other hand, we have $\Sq^2(\Sigma u\otimes\Sigma u^{n-3})=0$, implying $\Sq^2(H^n(\Sigma\R P^1\times\Sigma\R P^{n-1};\Z_2))=0$. Thus by Lemma \ref{Steenrod operation}, we get $[h,\tilde{g}]\ne 0$.

		\noindent(2) Suppose $n\equiv 1,2\mod 4$. Let $h=\tilde{g}\vert_{\Sigma\R P^2}$. Then we have $h^*(v_n)=0$ and $h^*(v_3)=\Sigma u^2$. By the Wu formula, we have $\Sq^3w_n=w_3w_n$, implying
		\[
		\Sq^3v_n=v_3v_n.
		\]
		On the other hand, we have $\Sq^3(\Sigma u^2\otimes\Sigma u^{n-4})=0$ and $\Sq^3(\Sigma u\otimes\Sigma u^{n-5})=0$, implying $\Sq^3(H^n(\Sigma\R P^2\times\Sigma\R P^{n-1};\Z_2))=0$. Thus by Lemma \ref{Steenrod operation}, the Whitehead product $[h,\tilde{g}]$ is non-trivial.

		By (1) and (2) together with Lemma \ref{non-commutative}, we obtain that the loop space of $\SU(n)/\SO(n)$ for $n\ge 3$ is not homotopy commutative, completing the proof.
	\end{proof}

	\begin{proposition}
		\label{BDI}
		The loop spaces of $\mathrm{BDI}=\SO(m+n)/\SO(m)\times\SO(n)$ for $m,n\ge 2$ is not homotopy commutative.
	\end{proposition}
	
	\begin{proof}
		We may assume $m\ge n$. The case $n=2$ is proved in \cite{KTT}, so we also assume $n\ge 3$. Consider the map $\bar{g}\colon\Sigma\R P^{n-1}\to B\SO(n)$ in the proof of Proposition \ref{AI}. Quite similarly to the proof of Proposition \ref{AI}, we can show that $[\bar{g},\bar{g}\vert_{\Sigma\R P^1}]\ne 0$ for $n\equiv 0,3\mod 4$ and $[\bar{g},\bar{g}\vert_{\Sigma\R P^2}]\ne 0$ for $n\equiv 1,2\mod 4$. Thus we obtain $[\bar{g},\bar{g}]\ne 0$.

		Since the the natural map $\iota\colon\SO(m+n)/\SO(m)\times\SO(n)\to B\SO(n)$ is an $n$-equivalence, the map $\bar{g}$ lifts to a map $\tilde{g}\colon\Sigma\R P^{n-1}\to\SO(m+n)/\SO(m)\times\SO(n)$ through $\iota$, up to homotopy. Then since $\iota_*([\tilde{g},\tilde{g}])=[\bar{g},\bar{g}]\ne 0$, we get $[\tilde{g},\tilde{g}]\ne 0$. Thus by Lemma \ref{non-commutative}, the loop space of $\SO(m+n)/\SO(m)\times\SO(n)$ for $m\ge n\ge 3$ is not homotopy commutative, completing the proof.
	\end{proof}

	It remains to consider $\mathrm{CII}$.
	\begin{proposition}
		\label{CII}
		The loop space of $\mathrm{CII}=\Sp(m+n)/\Sp(m)\times\Sp(n)$ for $m,n\ge 1$ is not homotopy commutative.
	\end{proposition}
	
	\begin{proof}
		Recall that the cohomology of $B\Sp(n)$ is given by
		\[
		H^*(B\Sp(n))=\Z[q_1,\ldots,q_n]
		\]
		where $q_i$ denotes the $i$-the symplectic Pontrjagin class. Let $Q_n$ denote the quaternionic quasi-projective space in the sense of James \cite{J}. Then we have
		\[
		H^*(Q_n)=\langle x_1,\ldots,x_n\rangle
		\]
		such that the natural map $g\colon Q_n\to\Sp(n)$ satisfies $g^*(\sigma(q_i))=x_i$, where $\sigma$ denotes the cohomology suspension. Let $\bar{g}\colon\Sigma Q_n\to B\Sp(n)$ be the adjoint of the map $g$. Then we get
		\begin{equation}
			\label{q_i}
			\bar{g}^*(q_i)=\Sigma x_i.
		\end{equation}
		We aim to show $[\bar{g},\bar{g}]\ne 0$.

		\noindent(1) Suppose that $n$ is divisible by an odd prime $p$. Since the natural map $c\colon B\Sp(n)\to B\SU(n)$ satisfies $c^*(c_{2i})=(-1)^iq_i$, the mod $p$ Wu formula in \cite{Sh} shows
		\[
		\mathcal{P}^1q_n=(-1)^\frac{p-1}{2}q_{\frac{p-1}{2}}q_n
		\]
		in mod $p$ cohomology. We also have that $\mathcal{P}^1q_{n-\frac{p-1}{2}}=0$, implying $\mathcal{P}^1(H^{4n}(\Sigma Q_{\frac{p-1}{2}}\times\Sigma Q_n;\Z_p))=0$ by \eqref{q_i}. Then by Lemma \ref{Steenrod operation}, we obtain $[\bar{g}\vert_{\Sigma Q_{\frac{p-1}{2}}},\bar{g}]\ne 0$, implying $[\bar{g},\bar{g}]\ne 0$.

		\noindent(2) Suppose that $n$ is a power of $2$. Quite similarly to the above case, we have
		\[
		\Sq^4q_n=q_1q_n
		\]
		in mod $2$ cohomology. We also have that $\Sq^4q_{n-1}$ is decomposable, implying $\Sq^4(H^{4n}(\Sigma Q_1\times\Sigma Q_n;\Z_2))=0$ by \eqref{q_i}, where $Q_1=S^3$. Then by Lemma \ref{Steenrod operation}, we obtain $[\bar{g}\vert_{\Sigma Q_1},\bar{g}]\ne 0$, implying $[\bar{g},\bar{g}]\ne 0$.

		For $\mathrm{CII}=\Sp(m+n)/\Sp(m)\times\Sp(n)$, we may assume $m\ge n$. Then the natural map $\iota\colon\mathrm{CII}\to B\Sp(n)$ is an $(4n+2)$-equivalence, implying that the map $\bar{g}$ lifts to a map $\tilde{g}\colon\Sigma Q_n\to\mathrm{CII}$ through $\iota$, up to homotopy. Since
		\[
		\iota_*([\tilde{g},\tilde{g}])=[\iota_*(\tilde{g}),\iota_*(\tilde{g})]=[\bar{g},\bar{g}]\ne 0,
		\]
		we get $[\tilde{g},\tilde{g}]\ne 0$. Thus by Lemma \ref{non-commutative}, the loop space of $\mathrm{CII}$ is not homotopy commutative, completing the proof.
	\end{proof}

	\begin{remark}
		We may prove Proposition \ref{CII} by using the result of Bott \cite{B} through the natural map $\mathrm{CII}\to B\Sp(n)$.
	\end{remark}

	\subsection{$\mathrm{EI}$, $\mathrm{FII}$ and $\G$}
	
	We consider the symmetric spaces $\mathrm{EI}$, $\mathrm{FII}$ and $\G$.

	\begin{proposition}
		\label{EI}
		The loop space of $\mathrm{EI}$ is not homotopy commutative.
	\end{proposition}
	
	\begin{proof}
		By \cite{I1}, the mod $5$ cohomology of $\mathrm{EI}$ is given by
		\[
		H^*(\mathrm{EI};\Z_5)=\Z_5[x_8]/(x_8^3)\otimes\Lambda(x_9,x_{17}).
		\]
		Recall that the mod $5$ cohomology of $B\mathrm{PSp}(4)$ and $B\E_6$ are given by
		\begin{alignat*}{3}
			H^*(B\mathrm{PSp}(4);\Z_5)&=\Z_5[q_1,q_2,q_3,q_4],&&\quad|q_i|=4i\\
			H^*(B\E_6;\Z_5)&=\Z_5[y_4,y_{12},y_{16},y_{20},y_{24},y_{32}],&&\quad|y_i|=i.
		\end{alignat*}
		We consider the Serre spectral sequence associated with a homotopy fibration
		\[
		\mathrm{EI}\xrightarrow{\iota}B\mathrm{PSp}(4)\to B\E_6.
		\]
		Then by degree reasons, we get $\iota^*(q_2)=x_8$. Then since $\mathcal{P}^1q_2=q_2^2$, we obtain $\mathcal{P}^1x_8=x_8^2$. Let $g\colon S^8\to\mathrm{EI}_{(5)}$ be a map detecting $x_8$, where $-_{(5)}$ stands for the $5$-localization. Then since $\mathcal{P}^1(H^*(S^8\times S^8;\Z_5))=0$, we can apply Lemma \ref{Steenrod operation} and get that the Whitehead product $[g,g]$ is non-trivial. Thus by Lemma \ref{non-commutative}, the proof is finished.
	\end{proof}

	\begin{proposition}
		\label{FII}
		The loop space of $\mathrm{FII}$ is not homotopy commutative.
	\end{proposition}
	
	\begin{proof}
		Since $\mathrm{FII}$ is the Cayley projective plane $\O P^2$, its mod $5$ cohomology is given by
		\begin{equation}
			\label{q}
			H^*(\mathrm{FII};\Z_5)=\Z_5[q]/(q^3),\quad|q|=8.
		\end{equation}
		Consider the homotopy fibration
		\[
		\F_4/\Spin(9)\xrightarrow{\iota}B\Spin(9)\to B\F_4.
		\]
		Since mod $5$ cohomology of $B\F_4$ and $B\Spin(9)$ are given by
		\begin{align*}
			H^*(B\F_4;\Z_5)&=\Z_5[x_4,x_{12},x_{16},x_{24}]\\
			H^*(B\Spin(9);\Z_5)&=\Z_5[p_1,p_2,p_3,p_4]
		\end{align*}
		where $|x_i|=i$ and $p_i$ is the $i$-th Pontrjagin class, the standard spectral sequence argument shows $q=\iota^*(p_2)$. Then since $\mathcal{P}^1p_2=p_2^2+2p_2p_1^2$, we have $\mathcal{P}^1q=q^2$. On the other hand, we have $\mathcal{P}^1(H^*(S^8\times S^8;\Z_5))=0$. Let $g\colon S^8\to\mathrm{FII}$ denote the bottom cell inclusion. Then by Lemma \ref{Steenrod operation}, we get $[g,g]\ne 0$. Thus by Lemma \ref{non-commutative}, the loop space of $\mathrm{FII}$ is not homotopy commutative.
	\end{proof}

	%\begin{proposition}
	%  \label{EIII}
	%  The loop space of $\mathrm{EIII}$ is not homotopy commutative.
	%\end{proposition}
	
	%\begin{proof}
	%  In \cite{TW}, Toda and Watanabe determined the integral cohomology of $\mathrm{EIII}$. In particular, its mod $5$ cohomology is given by
	%  \[
	%    H^*(\mathrm{EIII};\Z_5)=\Z_5[x_2,x_8]/(x_2^9+2x_2x_8^2,\,x_2^8x_8+x_8^3),\quad|x_i|=i.
	%  \]
	%  By \cite{Co}, there is an inclusion $\alpha\colon\mathrm{FII}\to\mathrm{EIII}$ whose homotopy fiber is $S^1$ in dimension $\le 16$. Then we get
	%  \[
	%    \alpha^*(w)=q
	%  \]
	%  where $q$ is as in \eqref{q}. Then since $\mathcal{P}^1q=q^2$ as in the proof of Proposition \ref{FII}, we get that $\mathcal{P}^1x_8$ includes $x_8^2$. By degree reasons, $\mathcal{P}^1x_8$ is decomposable. Moreover, we have $\mathcal{P}^1(H^*(S^8\times S^8;\Z_5))=0$. Thus by Lemma \ref{Steenrod operation}, we obtain $[\alpha\vert_{S^8},\alpha\vert_{S^8}]\ne 0$, so the loop space of $\E_6/\Spin(10)\cdot S^1$ is not homotopy commutative, as stated.
	%\end{proof}

	\begin{proposition}
		\label{G}
		The loop space of $\G$ is not homotopy commutative.
	\end{proposition}
	
	\begin{proof}
		Let $\iota\colon\G\to B\SO(4)$ denote the natural map. In \cite{BH}, the mod $2$ cohomology of $\G$ is given by
		\[
		H^*(\G;\Z_2)=\Z_2[x_2,x_3]/(x_2^3+x_3^2,x_2x_3),\quad\iota^*(w_i)=x_i
		\]
		where $w_i$ is the $i$-th Stiefel-Whitney class. Then by the Wu formula, we have $\Sq^1x_2=x_3$, implying that the $3$-skeleton of $\G$ is $M=S^2\cup_2e^3$. Let $g\colon M\to\G$ denote the inclusion. Then we have $g^*(x_i)=u_i$ for $i=2,3$, where $u_i$ is a generator of $H^i(M;\Z_2)\cong\Z_2$ for $i=2,3$. Clearly, we have $\Sq^2(H^*(S^2\times M;\Z_2))=0$. By the Wu formula, we also have $\Sq^2x_3=x_2x_3$. Then by Lemma \ref{Steenrod operation}, we get that the Whitehead product $[g\vert_{S^2},g]$ is non-trivial. Thus by Lemma \ref{non-commutative}, the proof is finished.
	\end{proof}

	%%%%% Section 2 %%%%%
	
	\section{Partial projective plane}
	
	In this section, we introduce a partial projective plane, and use it to prove the existence of a non-trivial Whitehead product. We apply this technique to $\mathrm{AII}$ and $\mathrm{EIV}$, to which the techniques in the previous sections do not apply.

	Let $X$ be a path-connected space. We say that a map $g\colon A\to X$ is a generating map if it induces an isomorphism
	\[
	g^*\colon QH^*(X)\xrightarrow{\cong}\widetilde{H}^*(A).
	\]
	For example, the map $h\colon\Sigma\C P^{n-1}\to\SU(n)$ defined in the proof of Proposition \ref{AI} is a generating map. A generating map is of particular importance in the study of the multiplicative structure of a localized Lie group \cite{KKTh,KKTs1,KKTs2,Th1,Th2}.

	Let $g\colon A\to X$ be a generating map such that $A=\Sigma B$. Assume that the Whitehead product $[g,g]\colon B\star B\to X$ is trivial. Then there is a homotopy commutative diagram
	\[
	\xymatrix{
		A\vee A\ar[r]^(.62){g+g}\ar[d]_{\rm incl}&X\ar@{=}[d]\\
		A\times A\ar[r]^(.62)\mu&X.
	}
	\]
	Fixing the map $\mu$, we define the partial projective plane of $X$, denoted by $\widehat{P}_2X$, by the cofiber of the Hopf construction
	\[
	H(\mu)\colon A\star A\to\Sigma X,\quad(x,y,t)\mapsto(\mu(x,y),t).
	\]
	If $X$ is an H-space and $\mu$ is the restriction of the multiplication of $X$, then $\widehat{P}_2X$ is a subspace of the projective plane of $X$. We show a condition for the existence of a non-trivial Whitehead product by using the partial projective plane. Let
	\[
	\Sq=1+\Sq^1+\Sq^2+\cdots.
	\]

	\begin{lemma}
		\label{projective plane}
		Let $X$ be a path-connected space satisfying the following conditions:
		\begin{enumerate}
			\item the mod $2$ cohomology of $X$ is generated by elements $x_1,\ldots,x_n$ of odd degree;
			
			\item $\Sq x_i$ is a linear combination of $x_1,\ldots,x_n$ for $i=1,\ldots,n$.
		\end{enumerate}
		If there is a generating map $g\colon A=\Sigma B\to X$ and $\min\{|x_1|,\ldots,|x_n|\}\ne 2^k-1$ for any $k\ge 1$, then the Whitehead product $[g,g]$ is non-trivial.
	\end{lemma}
	
	\begin{proof}
		We assume $[g,g]=0$ and deduce a contradiction. By assumption, we have the partial projective space $\widehat{P}_2X$, where we fix a map $\mu\colon A\times A\to X$ extending the map $\alpha+\alpha\colon A\vee A\to X$, up to homotopy. By degree reasons, each $x_i\in H^*(X;\Z_2)$ is primitive in the sense that
		\[
		\mu^*(x_i)=g^*(x_i)\otimes 1+1\otimes g^*(x_i).
		\]
		Quite similarly to \cite{BT}, we can see that
		\[
		H(\mu)^*(\Sigma x)=\Sigma\mu^*(x)-\Sigma(g^*(x)\otimes 1)-\Sigma(1\otimes g^*(x))
		\]
		for any $x\in H^*(X)$, where we identify $A\star A$ with $\Sigma A\wedge A$. Then we get $H(\mu)^*(\Sigma x_i)=0$. By definition, there is an exact sequence
		\[
		\cdots\to H^{*-1}(A\star A;\Z_2)\xrightarrow{\delta}H^*(\widehat{P}_2X;\Z_2)\xrightarrow{\iota^*}H^*(\Sigma X;\Z_2)\xrightarrow{H(\mu)^*}H^*(A\star A;\Z_2)\to\cdots
		\]
		where $\iota\colon\Sigma X\to\widehat{P}_2X$ denotes the inclusion. Since $H(\mu)^*(\Sigma x_i)=0$, there is $y_i\in H^*(\widehat{P}_2X)$ such that $\iota^*(y_i)=\Sigma x_i$. So by \cite{T1}, we get
		\[
		\delta(g^*(x_i)\star g^*(x_j))=y_iy_j.
		\]
		On the other hand, since $\widehat{P}_2X$ has category $\le2$, we have $y_iy_jy_k=0$ for any $i,j,k$. Then $H^*(\widehat{P}_2X;\Z_2)$ contains the truncated polynomial ring
		\[
		A=\Z_2[y_1,y_2,\ldots,y_n]/(y_1,y_2,\ldots,y_n)^3.
		\]
		Since $\Sq x_i$ is a linear combination of $x_1,\ldots,x_n$, $\Sq y_i$ is a linear combination of $y_1,\ldots,y_n$ modulo $\mathrm{Im}\,\delta$. Then since $\mathrm{Im}\,\delta$ is a vector space spanned by $y_iy_j$ for $1\le i\le j\le n$, we get
		\[
		\Sq(A)\subset A.
		\]
		Thus by \cite{T2}, we obtain that $\min\{|y_1|,|y_2|,\ldots,|y_n|\}$ must be $2^k$ for some $k\ge 1$, completing the proof because $|y_i|=|x_i|+1$.
	\end{proof}

	Now we consider the symmetric spaces $\mathrm{AII}$ and $\mathrm{EIV}$.

	\begin{proposition}
		\label{AII}
		The loop space of $\mathrm{AII}=\SU(2n)/\Sp(n)$ for $n\ge 2$ is not homotopy commutative.
	\end{proposition}
	
	\begin{proof}
		By \cite[Theorem 6.7]{MT} together with the Wu formula, the mod $2$ cohomology of $\mathrm{AII}$ is given by
		\[
		H^*(\mathrm{AII};\Z_2)=\Lambda(x_5,x_9,\ldots,x_{4n-3}),\quad|x_{4i+1}|=4i+1
		\]
		such that $\Sq x_{4i+1}$ is a linear combination of $x_5,\ldots,x_{4n-3}$ for $i=1,\ldots,n-1$. Mukai and Oka \cite{MO} showed that there is a generating map $g\colon\Sigma\H P^{n-1}\to\SU(2n)/\Sp(n)$. Then since $\min\{|x_5|,\ldots,|x_{4n-3}|\}=5$, it follows from Lemma \ref{projective plane} that the Whitehead product $[g,g]$ must be non-trivial. Thus the proof is finished by Lemma \ref{non-commutative}.
	\end{proof}

	\begin{proposition}
		\label{EIV}
		The loop space of $\mathrm{EIV}$ is not homotopy commutative.
	\end{proposition}
	
	\begin{proof}
		In \cite{A}, Araki showed there is a cell decomposition $\mathrm{EIV}=(\Sigma\O P^2)\cup e^{26}$, and the cohomology of $\mathrm{EIV}$ is given by
		\[
		H^*(\mathrm{EIV};\Z)=\Lambda(x_9,x_{17}),\quad|x_i|=i.
		\]
		Then the inclusion $g\colon\Sigma\O P^2\to\mathrm{EIV}$ is a generating map. By the above cell decomposition, we have $\Sq x_9=x_9+x_{17}$, so by the Adem relation, we have $\Sq x_{17}=x_{17}$ too. Thus by Lemma \ref{projective plane}, the Whitehead product $[g,g]$ must be non-trivial, so the proof is finished by Lemma \ref{non-commutative}.
	\end{proof}

	Finally, we are ready to prove Theorem \ref{main}.
	\begin{proof}
		[Proof of Theorem \ref{main}]
		Combine \cite[Theorem 1.1]{KTT} and Propositions \ref{EII}, \ref{EV-EVIII}, \ref{EVI-EIX-FI}, \ref{AI}, \ref{BDI}, \ref{CII}, \ref{EI}, \ref{FII}, \ref{G}, \ref{AII}, and \ref{EIV}.
	\end{proof}


\begin{thebibliography}{99}
		\bibitem{A} S. Araki, cohomology modulo 2 of the compact exceptional groups $E_6$ and $E_7$, J. Math. Osaka City Univ. \textbf{12} (1961), 43-65.
		
		\bibitem{B} R. Bott, A note on the Samelson products in the classical groups, Comment. Math. Helv. \textbf{34} (1960), 249-256.
		
		\bibitem{BT} W. Browder and E. Thomas, On the projective plane of an $H$-space, Illinois J. Math. \textbf{7} (1963), no. 3, 492-502.
		
		\bibitem{BH} A. Borel and F. Hirzebruch, Characteristic classes and homogeneous spaces. I, Amer. J. Math. \textbf{80} (1958), 458-538.
		
		\bibitem{Ca} E. Cartan, Sur les nombres de Betti des espaces de groupes close, C.R. Acad. Sci. Paris \textbf{187} (1928), 196-197.
		
		%\bibitem{Co} L. Conlon, On the topology of EIII and EIV, Proc. Amer. Math. Soc. \textbf{16} (1965), no. 4, 575-581.
		
		\bibitem{FHT} Y. F\'elix, S. Halperin, and J.-C. Thomas, Rational Homotopy Theory, Graduate Texts in Mathematics \textbf{205}, Springer-Verlag, New York, 2001.
		
		\bibitem{Ga} T. Ganea, On the loop spaces of projective spaces, J. Math. Mech. \textbf{16} (1967), 853-858.
		
		\bibitem{Go} M. Golasi\'{n}ski, Homotopy nilpotency of some homogeneous spaces, Manuscripta Math. \textbf{167} (2022), 245-261.
		
		\bibitem{HKMO} S. Hasui, D. Kishimoto, T. Miyauchi, and A. Ohsita, Samelson products in quasi-$p$-regular exceptional Lie groups, Homology Homotopy Appl. \textbf{20} (2018), no. 1, 185-208.
		
		\bibitem{H} J.R. Hubbuck, On homotopy commutative H-spaces, Topology \textbf{8} (1969), 119-126.
		
		\bibitem{I2} K. Ishitoya, Integral cohomology ring of the symmetric space EII, J. Math. Kyoto Univ. \textbf{17} (1977), no. 2, 375-397.
		
		\bibitem{I1} K. Ishitoya, Cohomology of the symmetric space EI, Proc. Japan Acad. Ser. A Math. Sci. \textbf{53} (1977), no. 2, 56-60.
		
		\bibitem{IT} K. Ishitoya and H. Toda, On the cohomology of irreducible symmetric spaces of exceptional type, J. Math. Kyoto Univ. \textbf{17} (1977), no. 2, 225-243
		
		\bibitem{J} I.M. James, Spaces associated with Stiefel manifolds, Proc. London Math. Soc., \textbf{9} (1959), 115-140.
		
		\bibitem{KTT} D. Kishimoto, M. Takeda, and Y. Tong, Homotopy commutativity in Hermitian symmetric spaces, Glasg. Math. J. \textbf{64} (2022), no. 3, 746-752.
		
		\bibitem{KKTh} D. Kishimoto, A. Kono, and S. Theriault, Homotopy commutativity in $p$-localized gauge groups, Proc. Roy. Soc. Edinburgh Sect. A \textbf{143} (2013), no. 4, 851-870.
		
		\bibitem{KKTs1} D. Kishimoto, A. Kono, and M. Tsutaya, Mod p decompositions of gauge groups, Algebr. Geom. Topol. \textbf{13} (2013), no. 3, 1757-1778.
		
		\bibitem{KKTs2} D. Kishimoto, A. Kono, and M. Tsutaya, On $p$-local homotopy types of gauge groups, Proc. Roy. Soc. Edinburgh Sect. A \textbf{144} (2014), no. 1, 149-160.
		
		\bibitem{MN} T. Miller and J. Neisendorfer, Formal and coformal spaces, Illinois J. Math. \textbf{22} (1978), no. 4, 565-580.
		
		\bibitem{MT} M. Mimura and H. Toda, Topology of Lie groups, I and II, Translations of Mathematical Monographs \textbf{91} (AMS, Providence, RI, 1991).
		
		\bibitem{MO} J. Mukai and S. Oka, A note on the quaternionic quasi-projective space, Mem. Fac. Sci., Kyushuu Univ., Ser. A \textbf{38} (1984), 277-284.
		
		\bibitem{N1} M. Nakagawa, The integral cohomology ring of $E_8/T^1\cdot E_7$, arXiv:math.AT/09114793.
		
		\bibitem{N2} M. Nakagawa, The integral cohomology ring of $E_7/T$, J. Math. Kyoto Univ. \textbf{41} (2001), no. 2, 303-321.
		
		\bibitem{Sh} P.B. Shay, mod $p$ Wu formulas for the Steenrod algebra and the Dyer-Lashof algebra, Proc. Amer. Math. Soc. \textbf{63} (1977), no. 2, 339-347.
		
		%\bibitem{S} J. Stasheff, On homotopy Abelian H-spaces, Proc. Cambridge Philos. Soc. \textbf{57} (1961), 734-745.
		
		\bibitem{Th1} S.D. Theriault, The  H-structure of low-rank torsion free  H-spaces, Q. J. Math. \textbf{56} (2005), no. 3, 403-415.
		
		\bibitem{Th2} S.D. Theriault, Odd primary homotopy decompositions of gauge groups, Algebr. Geom. Topol. \textbf{10} (2010), no. 1, 535-564.
		
		\bibitem{T1} E. Thomas, On functional cup-products and the transgression operator, Arch. Math \textbf{12} (1961), 435-444.
		
		\bibitem{T2} E. Thomas, Steenrod squares and H-spaces, Ann. of Math. \textbf{77} (1963), no. 2, 306-317.
		
		%  \bibitem{TW} H. Toda and T. Watanabe, The integral cohomology ring of $F_4/T$ and $E_6/T$, J. Math. Kyoto Univ. \textbf{14} (1974), no. 2, 257-286.
		
		\bibitem{W} T. Watanabe, The integral cohomology ring of the symmetric space EVII, J. Math. Kyoto Univ. \textbf{15} (1975), no. 2, 363-385.
		
		\bibitem{Wh} J.H.C. Whitehead, On the groups $\pi_r(V_{n,m})$ and sphere bundles, Proc. London Math. Soc. \textbf{48} (1944), 243-291.
		
		\bibitem{Y} I. Yokota, On the cellular decomposition of unitary groups, J. Inst. Polytech. Osaka City Univ. Ser. A \textbf{7} (1956), 39-49. 
		
	\end{thebibliography}
\end{document}